\documentclass[11pt]{amsart}

\usepackage[T1]{fontenc}
\usepackage{lmodern}
\usepackage{microtype}
\usepackage{amsmath,amssymb}
\usepackage[hidelinks]{hyperref}
\usepackage[nameinlink,noabbrev]{cleveref}

\hypersetup{
  pdftitle={The missing two-point inequality in Weissler's conjecture},
  pdfsubject={Complex hypercontractivity on the Hamming cube},
  pdfkeywords={complex hypercontractivity, Hamming cube, two-point inequality, matrix flow}
}

\numberwithin{equation}{section}
\allowdisplaybreaks

\newtheorem{theorem}{Theorem}[section]
\newtheorem{proposition}[theorem]{Proposition}
\newtheorem{lemma}[theorem]{Lemma}

\newcommand{\C}{\mathbb C}
\newcommand{\R}{\mathbb R}
\newcommand{\Om}{\Omega}
\newcommand{\Rea}{\operatorname{Re}}
\newcommand{\Ima}{\operatorname{Im}}
\newcommand{\sech}{\operatorname{sech}}
\newcommand{\diag}{\operatorname{diag}}
\newcommand{\tr}{\operatorname{tr}}

\newcommand{\dd}{\,\mathrm d}
\newcommand{\norm}[1]{\lVert #1\rVert}
\newcommand{\abs}[1]{\lvert #1\rvert}
\newcommand{\ip}[2]{\langle #1,#2\rangle}

\title[The missing two-point inequality]
{The missing two-point inequality in Weissler's conjecture}

\author{Yi C. Huang}
\address{School of Mathematical Sciences, Nanjing Normal University, Nanjing, China}
\email{yi.huang.analysis@gmail.com}

\author{Paata Ivanisvili}
\address{Department of Mathematics, University of California, Irvine,
Irvine, CA 92697, USA}
\email{pivanisv@uci.edu}

\subjclass[2020]{42B35, 42A38, 47A30, 60E15}

\keywords{Complex hypercontractivity, Hamming cube, noise operator,
two-point inequality, matrix flow}

\begin{document}

\begin{abstract}
Weissler's conjectured characterization of complex hypercontractivity on the
Hamming cube remained open in the ranges $2<p\le q<3$ and
$\frac32<p\le q<2$.  Ivanisvili--Nazarov established the diagonal cases.
We prove the remaining strict off-diagonal cases through the corresponding
two-point inequality. Our argument gives a self-contained proof for every
$2<p\le q<\infty$ and, by duality, for every $1<p\le q<2$, including
the previously known diagonal cases. The proof is based on a flow argument.
\end{abstract}

\maketitle

\section{Statement and proof idea}

Hypercontractivity emerged in constructive quantum field theory and Gaussian analysis. Early work of Nelson, Glimm, and Segal led to the sharp real Gaussian theorem, while Bonami established the corresponding real inequality on finite product groups; see \cite{nelson1966,glimm1968,segal1970,bonami1970,nelson1973}. Gross's logarithmic Sobolev method and Beckner's sharp Hausdorff--Young inequality  made hypercontractivity a central tool in analysis and probability \cite{gross1975,beckner1975}. Modern accounts and applications to Boolean functions may be found in \cite{bakry2014,odonnell2014}.
The complex theory is substantially more delicate because the positivity
and Markov structure available for real parameters disappear.
Beckner~\cite{beckner1975} proved the two-point inequality for
$q=p'$ and $z=i\sqrt{p-1}$, $1<p\le2$, where $p'=p/(p-1)$.
By a central limit argument he obtained the corresponding Gaussian
inequality, leading to the sharp Hausdorff--Young inequality.
Coifman--Cwikel--Rochberg--Sagher--Weiss~\cite{coifman1979}
subsequently established the Gaussian $L^p\to L^{p'}$ inequality
throughout the full admissible complex-parameter region.
Weissler~\cite{Weissler} obtained the general sharp discrete criterion
except in the ranges $2<p\le q<3$ and $3/2<p\le q<2$. Epperson proposed an early proof of the Gaussian counterpart; that argument contained a gap, while complete Gaussian proofs were supplied by Lieb and Janson. A later heat-flow presentation is due to Hu; see \cite{epperson1989,lieb1990,janson1997,hu2016}. These Gaussian results do not by themselves yield the discrete two-point inequality on the cube.

Let $\mu_n$ be the uniform probability measure on $\{-1,1\}^n$.  If
\[
  f=\sum_{S\subseteq[n]}\widehat f(S)w_S,
  \qquad
  w_S(x)=\prod_{j\in S}x_j,
\]
the complex noise operator is
\[
  T_zf=\sum_{S\subseteq[n]}z^{\abs S}\widehat f(S)w_S.
\]
Let  $1\le p\le q<\infty$. 
We are interested in the inequality of the form 
\begin{align}\label{hyp21}
   \|T_{z}f\|_{q} \leq \|f\|_{p}  \quad \text{for all} \quad f : \{-1,1\}^{n} \to \mathbb{C}
\end{align}
 and all $n\geq 1$. 
Set
\begin{equation}\label{eq:Omega-definition}
  \Om_{p,q}:=
  \left\{z\in\C:
  \abs{p-2-z^2(q-2)}\le p-q\abs z^2\right\}.
\end{equation}
The necessity of the condition $z\in\Om_{p,q}$ is trivial and was known before (see for example \cite{Weissler}). However, sufficiency of the condition $z\in\Om_{p,q}$ was known except in the ranges
\[
  2<p\le q<3,
  \qquad
  \frac32<p\le q<2.
\]
The diagonal cases $p=q$ in these ranges were established more recently by
Ivanisvili--Nazarov~\cite{IvanisviliNazarov}. They also announced an
off-diagonal proof involving the verification of nonnegativity of
high-degree polynomials, but that proof has not appeared publicly.
The proof presented here is different and uses a unified flow argument.
It is standard that the dimension-free statement \eqref{hyp21} is equivalent (see~\cite{Weissler,IvanisviliNazarov}) to the following inequality on the
two-point space:
\begin{equation}\label{eq:two-point-main}
  \left(\frac{\abs{a+zb}^q+\abs{a-zb}^q}{2}\right)^{1/q}
  \le
  \left(\frac{\abs{a+b}^p+\abs{a-b}^p}{2}\right)^{1/p}
  \qquad \; a,b\in\C.
\end{equation}
Thus only the two-point inequality
needs to be proved.

\begin{theorem}\label{thm:two-point}
Let $2<p\le q<\infty$ and let $z\in\Om_{p,q}$. Then
\eqref{eq:two-point-main} holds for every $a,b\in\C$.
\end{theorem}

The new range needed for Weissler's conjecture is $2<p<q<3$.
The proof below also recovers the diagonal cases without using the
diagonal theorem as an input. Standard tensorization
\cite{Weissler,IvanisviliNazarov} yields dimension-free complex
hypercontractivity, and Step~5 gives the entire dual range
$1<p\le q<2$.

\clearpage
\subsection*{A guide to the proof}

\emph{Normalizing the two-point function.}
Every function on $\{-1,1\}$ has the form $f(x)=a+bx$, with
$a,b\in\C$, and $T_zf(x)=a+zbx$. Thus $T_z$ multiplies the
coefficient of $x$ by $z$. If $a\ne0$, set $w=b/a$ and define
\[
  \mathcal N_s(w):=
  \left(\frac{|1+w|^s+|1-w|^s}{2}\right)^{1/s},
  \qquad H_s(w):=\log\mathcal N_s(w).
\]
Then $\|f\|_p=|a|\mathcal N_p(w)$ and
$\|T_zf\|_q=|a|\mathcal N_q(zw)$.  After cancelling
$|a|$, our goal is therefore
\[
  \mathcal N_q(zw)\le\mathcal N_p(w)
  \qquad\text{for every }w\in\C.
\]
The case $a=0$ follows by continuity in $a$. The difficulty is that
$\mathcal N_s(w)$ depends nonlinearly on $w$, while multiplication by a
complex number changes both the modulus and the argument of $w$.

\emph{What the inequality says for small $w$.}
Writing $w=u+iv$, a Taylor expansion gives
\[
  H_s(w)=\tfrac12\bigl((s-1)u^2+v^2\bigr)+O(|w|^4),
  \qquad |w|\longrightarrow0.
\]
The real and imaginary parts therefore contribute with different
weights. Put $D_s=\diag(s-1,1)$, and let $M_z$ be the real $2\times2$
matrix representing multiplication by $z$ on $\C=\R^2$.
For the desired inequality to hold for all sufficiently small $w$, it
is necessary that
\[
  M_z^{\mathsf T}D_qM_z\preceq D_p.
\]
Here $A\preceq B$ means that $B-A$ is positive semidefinite.
Step~1 shows that this matrix inequality is exactly the defining
condition $z\in\Om_{p,q}$. It controls the quadratic terms in the
Taylor expansion. The main issue is why the same condition should
control the full expressions for arbitrary $w$.

\emph{A path with prescribed endpoints.}
First fix $z$ on the boundary of $\Om_{p,q}$; interior points will be
handled afterward. We seek a path consisting of a real exponent
$s(t)>2$ and a nonzero complex multiplier $\mu(t)$. Rescaling the
parameter interval to $0\le t\le1$, we require
\[
  s(0)=p,\qquad s(1)=q,\qquad
  \mu(0)=1,\qquad \mu(1)=z.
\]
At time $t$, the two-point function is $x\mapsto1+\mu(t)wx$, and we
measure it in $L^{s(t)}$. For each fixed $w$, define
\[
  F_w(t):=H_{s(t)}(\mu(t)w).
\]
If we can arrange that $F_w'(t)\le0$ for every $w$, then
$F_w(1)\le F_w(0)$ gives precisely
$\mathcal N_q(zw)\le\mathcal N_p(w)$. Thus the purpose of the flow
is to interpolate between the two sides of the desired inequality
while keeping this logarithmic norm nonincreasing.

\emph{How the path is constrained.}
All derivatives below are with respect to $t$. We describe the rates
of change along the path by real-valued functions $c(t)$, $\beta(t)$,
and $\kappa(t)$, defined by
\[
  \frac{\mu'(t)}{\mu(t)}=-\frac{c(t)}2+i\beta(t),
  \qquad s'(t)=\kappa(t).
\]
In particular, $(\log|\mu(t)|)'=-c(t)/2$, so $c(t)\ge0$ means that
the modulus of the multiplier does not increase. The function $\beta(t)$ is
the rate of change of its argument, and $\kappa(t)$ is the rate of
change of the exponent.

To understand which rates to choose, first differentiate the quadratic
expression from the preceding Taylor expansion. Writing
$s=s(t)$, $c=c(t)$, $\beta=\beta(t)$, and $\kappa=\kappa(t)$, set
\[
  X(t):=
  \begin{pmatrix}
    (s-1)c-\kappa&(s-2)\beta\\
    (s-2)\beta&c
  \end{pmatrix}.
\]
A direct calculation gives
\[
  \frac{d}{dt}\left(M_{\mu(t)}^{\mathsf T}D_{s(t)}M_{\mu(t)}\right)
  =-M_{\mu(t)}^{\mathsf T}X(t)M_{\mu(t)}.
\]
Thus $X(t)\succeq0$ is exactly the condition that makes the quadratic
expression nonincreasing along the path for every real and imaginary
perturbation. This explains both the origin of $X(t)$ and the
condition imposed on it.

The main analytic result of Step~2 is that this condition also makes
the \emph{full} logarithmic norm nonincreasing, for every $w\in\C$.
More precisely, a derivative calculation gives an explicit real
symmetric matrix $K_s(\omega)$, depending nonlinearly on the complex
argument $\omega$, such that
\[
  F_w'(t)
  =-\tr\left(K_{s(t)}(\mu(t)w)X(t)\right).
\]
Lemma~\ref{lem:nonlinear-K} proves that $K_s(\omega)\succeq0$ for
every $s>2$ and $\omega\in\C$. The trace of a product of two positive
semidefinite matrices is nonnegative. Therefore, once we have chosen
a path with $X(t)\succeq0$, the displayed identity gives
$F_w'(t)\le0$ throughout that path. This is the step that upgrades
the information from small $w$ to an inequality for all $w$.

\emph{Why the nonlinear matrix is positive.}
The useful coordinates for proving this matrix lemma are
$w=\tanh(\alpha+i\vartheta)$, with $\alpha,\vartheta\in\R$.
Away from $w=\pm1$, they are obtained from
\[
  \frac{1+w}{1-w}=e^{2\alpha+2i\vartheta}.
\]
Thus $\alpha$ records the ratio of the magnitudes of $1+w$ and
$1-w$, while $\vartheta$ records their relative argument. The two-point
norm takes the simpler form
\[
  H_s\bigl(\tanh(\alpha+i\vartheta)\bigr)
  =\frac1s\log\cosh(s\alpha)-\log|\cosh(\alpha+i\vartheta)|.
\]
In these coordinates, the determinant of $K_s(w)$ becomes a positive
factor times a quadratic polynomial in $\cos(2\vartheta)$.
Two elementary hyperbolic inequalities locate the minimum of this
quadratic on $[-1,1]$ and show that it is nonnegative. Together with
the diagonal-entry calculation, this proves $K_s(w)\succeq0$.
The points $w=\pm1$ follow by continuity.

\emph{Reaching every boundary multiplier.}
It remains to construct paths satisfying $X(t)\succeq0$ with the
prescribed endpoints, for every $z\in\partial\Om_{p,q}$.
Sign and conjugation symmetries reduce the construction to the
first-quadrant boundary. The identity multiplier when $p=q$ is
immediate. Step~3 parametrizes the boundary by a real parameter $m$
and two branches. They start on the positive real and positive
imaginary axes, respectively, and meet at a common boundary point.
For $m>0$, the construction gives explicit paths to the
corresponding boundary points; the remaining endpoints are obtained
by continuity as $m\to0$. Along the nonconstant paths with $m>0$,
the diagonal entries of $X(t)$ are positive and its determinant is
zero. This verifies $X(t)\succeq0$ directly.

The two branches have a simple interpretation in terms of the
exponent. For one, $s(t)$ is nondecreasing from $p$ to $q$. For the other,
$s(t)$ first decreases from $p$ to an intermediate exponent
$s_*\in(2,p]$, and then increases to $q$; the decreasing part can
have zero length. The limiting boundary endpoints are obtained by
continuity. Throughout either path, $\mu(t)$ changes from $1$ to
the chosen boundary point $z$.

The important monotonicity is that of $F_w(t)$; the exponent $s(t)$
itself need not be monotone. Indeed, the matrix argument above places
no sign restriction on $\kappa(t)=s'(t)$. Allowing the exponent to
decrease on part of the path lets the explicit construction reach
both boundary branches. The parametrization stays smooth when the
exponent reaches its minimum and starts increasing. When $p=q$,
the exponent can start at $p$, decrease, and return to $p$, while
the multiplier changes from $1$ to $z$. The same argument therefore
also proves the diagonal theorem.

Integrating $F_w'(t)\le0$ now proves the two-point inequality for
$z$ on the boundary of $\Om_{p,q}$. Interior points are easier:
write $z=r z_0$, where $z_0\in\partial\Om_{p,q}$ and $0\le r<1$.
Multiplication of the coefficient by the additional real factor $r$
averages the two function values and cannot increase their $L^q$
norm. This gives the inequality throughout $\Om_{p,q}$.
Standard tensorization then passes from two points to the Hamming
cube, and operator duality gives the range below two. The proof thus
rests on two ingredients: positivity of $K_s(w)$ for all $w$, and
paths with $X(t)\succeq0$ that reach every boundary point.

\section{Detailed step-by-step proof}\label{sec:detailed-proof}

\subsection{Step 1: the matrix form and the radial boundary}

Identify $\C$ with $\R^2$ and define
\begin{equation}\label{eq:D-M-definitions}
  D_s:=\begin{pmatrix}s-1&0\\0&1\end{pmatrix},
  \qquad
  M_z:=\begin{pmatrix}\Rea z&-\Ima z\\ \Ima z&\Rea z\end{pmatrix}.
\end{equation}
Thus $M_z$ is real multiplication by $z$.

\begin{lemma}[Matrix form of admissibility]\label{lem:matrix-domain}
For $1\le p\le q<\infty$ and $z\in\C$,
\begin{equation}\label{eq:matrix-equivalence}
  z\in\Om_{p,q}
  \quad\Longleftrightarrow\quad
  M_z^{\mathsf T}D_qM_z\preceq D_p.
\end{equation}
\end{lemma}

\begin{proof}
Write $z=x+iy$ and set $H=D_p-M_z^{\mathsf T}D_qM_z$.  Direct
multiplication gives
\begin{equation}\label{eq:H-explicit}
  H=
  \begin{pmatrix}
    p-1-(q-1)x^2-y^2 & (q-2)xy\\
    (q-2)xy & 1-x^2-(q-1)y^2
  \end{pmatrix}.
\end{equation}
Hence
\begin{align}
  \tr H&=p-q\abs z^2,\label{eq:H-trace}\\
  (H_{11}-H_{22})+2iH_{12}
  &=p-2-(q-2)\overline z^{\,2}.\label{eq:H-traceless}
\end{align}

For a real symmetric $2\times2$ matrix $H$, its two eigenvalues are
\[
  \frac12\left(
    \tr H\pm
    \sqrt{(H_{11}-H_{22})^2+4H_{12}^2}
  \right).
\]
Consequently,
\[
  H\succeq0
  \quad \text{if and only if}\quad
  \tr H\ge
  \sqrt{(H_{11}-H_{22})^2+4H_{12}^2}.
\]
Substitution of
\eqref{eq:H-trace}--\eqref{eq:H-traceless} therefore gives exactly
\eqref{eq:Omega-definition}.
\end{proof}

From now on let
\begin{equation}\label{eq:PQ-definition}
  P:=p-1>1,
  \qquad
  Q:=q-1\ge P.
\end{equation}
For $z=\rho e^{i\theta}$, expansion of the determinant in
\eqref{eq:H-explicit} gives
\begin{equation}\label{eq:det-polar}
  \det\bigl(D_p-M_z^{\mathsf T}D_qM_z\bigr)
  =P-A_{P,Q}(\theta)\rho^2+Q\rho^4,
\end{equation}
where
\begin{equation}\label{eq:A-PQ}
  A_{P,Q}(\theta)
  =(PQ+1)\sin^2\theta+(P+Q)\cos^2\theta.
\end{equation}

\begin{lemma}[Radial boundary]\label{lem:radial-boundary}
In every direction $e^{i\theta}$ there is a unique radial boundary point
$\rho_*(\theta)e^{i\theta}$ of $\Om_{p,q}$.  Its squared radius is the
smaller root of
\begin{equation}\label{eq:boundary-quadratic}
  Qy^2-A_{P,Q}(\theta)y+P=0,
  \qquad y=\rho^2,
\end{equation}
where the roots are distinct unless $P=Q$ and $\sin\theta=0$.
\end{lemma}

\begin{proof}
Put
$B_\theta=M_{e^{i\theta}}^{\mathsf T}D_qM_{e^{i\theta}}$.  Since
$B_\theta$ is positive definite,
\[
  D_p-\rho^2B_\theta
\]
is strictly decreasing in the Loewner order as a function of $\rho^2$.
It is positive definite at $\rho=0$ and eventually negative definite.
Thus admissible radii form an interval $[0,\rho_*(\theta)]$, and the
determinant vanishes at its positive endpoint.  Formula
\eqref{eq:det-polar} gives \eqref{eq:boundary-quadratic}.  Moreover,
$A_{P,Q}(\theta)$ is a convex combination of $P+Q$ and $PQ+1$, and both
numbers are at least $2\sqrt{PQ}$ because $1<P\le Q$.
Equality is possible only for $P=Q$ and $\sin\theta=0$.
The first zero of the determinant is the smaller root, including the
double root in this exceptional case.
\end{proof}

For $R>1$, set
\begin{equation}\label{eq:tau-R}
  \tau_R:=\frac12\log R.
\end{equation}
Then
\begin{equation}\label{eq:tau-identities}
  \cosh\tau_R=\frac{R+1}{2\sqrt R},
  \qquad
  \sinh\tau_R=\frac{R-1}{2\sqrt R},
\end{equation}
and \eqref{eq:A-PQ} becomes
\begin{equation}\label{eq:A-hyperbolic}
  \frac{A_{P,Q}(\theta)}{2\sqrt{PQ}}
  =\cosh\tau_P\cosh\tau_Q
   -\sinh\tau_P\sinh\tau_Q\cos(2\theta).
\end{equation}

\subsection{Step 2: the infinitesimal two-point inequality}

For $s>2$ and $w\in\C$, define
\begin{equation}\label{eq:Hs-definition}
  \mathcal N_s(w):=
  \left(\frac{\abs{1+w}^s+\abs{1-w}^s}{2}\right)^{1/s},
  \qquad
  H_s(w):=\log\mathcal N_s(w).
\end{equation}
Write $w=x+iy$ and encode the ordinary real gradient as
\begin{equation}\label{eq:Gs-definition}
  G_s(w):=\partial_xH_s(w)+i\partial_yH_s(w).
\end{equation}
Finally, define the real-valued quantities
\begin{equation}\label{eq:ERI}
  \mathcal E_s(w):=\partial_sH_s(w),\qquad
  \mathcal R_s(w):=\Rea\!\bigl(\overline wG_s(w)\bigr),\qquad
  \mathcal I_s(w):=\Ima\!\bigl(\overline wG_s(w)\bigr).
\end{equation}
If $w=re^{i\phi}\ne0$, then
\[
  \mathcal R_s(w)=r\,\partial_rH_s(re^{i\phi}),
  \qquad
  \mathcal I_s(w)=\partial_\phi H_s(re^{i\phi}).
\]
Thus $\mathcal R_s$ is the logarithmic-radial derivative, whereas
$\mathcal I_s$ is the angular derivative.  The function $(s,w)\mapsto H_s(w)$ is continuously
differentiable on $(2,\infty)\times\C$, including at $w=\pm1$: this follows
from the continuous differentiability of $(s,u)\mapsto\abs u^s$ for $s>2$.

For $s>2$ and $c,\beta,\kappa\in\R$, put
\begin{equation}\label{eq:X-matrix}
  X_s(c,\beta,\kappa):=
  \begin{pmatrix}
    (s-1)c-\kappa & (s-2)\beta\\
    (s-2)\beta & c
  \end{pmatrix}.
\end{equation}
The first-order admissibility condition is
$X_s(c,\beta,\kappa)\succeq0$. Indeed, if
\[
  s(t)=s+\kappa t,
  \qquad
  \mu(t)=e^{(-c/2+i\beta)t},
  \qquad
  A=\begin{pmatrix}-c/2&-\beta\\ \beta&-c/2\end{pmatrix},
\]
then $M_{\mu(t)}=I+tA+O(t^2)$, and
\begin{align}
  D_s-M_{\mu(t)}^{\mathsf T}D_{s(t)}M_{\mu(t)}
  &=tX_s(c,\beta,\kappa)+O(t^2).
  \label{eq:first-order-expansion}
\end{align}
Indeed, the coefficient of $t$ is
$-(A^{\mathsf T}D_s+D_sA+\diag(\kappa,0))=X_s(c,\beta,\kappa)$.
There is no sign restriction on $\kappa$: the exponent may move in either
direction.

\begin{theorem}[Infinitesimal contraction]\label{thm:infinitesimal}
If $s>2$, $c,\beta,\kappa\in\R$, and
\begin{equation}\label{eq:X-psd}
  X_s(c,\beta,\kappa)\succeq0,
\end{equation}
then, for every $w\in\C$,
\begin{equation}\label{eq:inf-derivative}
  \left.\frac{\dd}{\dd t}\right|_{t=0}
  H_{s+\kappa t}\!\left(e^{(-c/2+i\beta)t}w\right)
  \le0.
\end{equation}
\end{theorem}

The proof rests on the following nonlinear matrix.

\begin{lemma}[Nonlinear dual matrix]\label{lem:nonlinear-K}
For every $s>2$ and $w\in\C$, the matrix
\begin{equation}\label{eq:K-matrix}
  K_s(w):=
  \begin{pmatrix}
    \mathcal E_s(w) & -\dfrac{\mathcal I_s(w)}{2(s-2)}\\[7pt]
    -\dfrac{\mathcal I_s(w)}{2(s-2)} &
    \dfrac{\mathcal R_s(w)}2-(s-1)\mathcal E_s(w)
  \end{pmatrix}
\end{equation}
is positive semidefinite.
\end{lemma}

We prove the lemma in three elementary steps.  Define
\begin{equation}\label{eq:L-definition}
  \mathcal L(x):=x\tanh x-\log\cosh x,
  \qquad x\ge0.
\end{equation}
Then
\begin{equation}\label{eq:L-derivative}
  \mathcal L(0)=0,
  \qquad
  \mathcal L'(x)=x\sech^2x\ge0.
\end{equation}

\begin{lemma}[Three bounds for $\mathcal L$]\label{lem:L-bounds}
For every $x\ge0$,
\begin{align}
  \mathcal L(x)&\le1-\sech x,\label{eq:L-upper}\\
  2\mathcal L(x)\cosh x&\ge x^2,\label{eq:L-lower-one}\\
  2\mathcal L(x)&\ge\tanh^2x.\label{eq:L-lower-two}
\end{align}
\end{lemma}

\begin{proof}
For \eqref{eq:L-upper},
\[
  \frac{\dd}{\dd x}\bigl(1-\sech x-\mathcal L(x)\bigr)
  =\frac{\sinh x-x}{\cosh^2x}\ge0,
\]
and the difference vanishes at zero.  For \eqref{eq:L-lower-one}, set
\[
  F(x):=\mathcal L(x)-\frac{x^2}{2\cosh x}.
\]
Then
\begin{align*}
  F'(x)
  &=x\sech^2x-\frac{x}{\cosh x}
    +\frac{x^2\sinh x}{2\cosh^2x}\\
  &=x\sech^2x
    \left(1-\cosh x+\frac{x}{2}\sinh x\right).
\end{align*}
Since
\[
  \cosh x-1=\sinh x\tanh(x/2)
  \quad\text{and}\quad
  x/2\ge\tanh(x/2),
\]
we have $F'(x)\ge0$, while $F(0)=0$.  Finally,
\[
  \frac{\dd}{\dd x}\bigl(2\mathcal L(x)-\tanh^2x\bigr)
  =2\sech^2x\,(x-\tanh x)\ge0,
\]
and this difference also vanishes at zero.
\end{proof}

\begin{lemma}[Two hyperbolic inequalities]\label{lem:two-hyperbolic}
For $x\ge0$ and $0\le r\le1$,
\begin{align}
  1-\frac{\cosh(rx)}{\cosh x}
  &\ge(1-r^2)\mathcal L(x),\label{eq:hyp-one}\\
  2r^2\mathcal L(x)\cosh(rx)\cosh x
  &\ge\sinh^2(rx).\label{eq:hyp-two}
\end{align}
\end{lemma}

\begin{proof}
For $0\le r\le1$, comparison of the power series gives
\[
  \cosh(rx)\le 1+r^2(\cosh x-1).
\]
Consequently,
\[
  1-\frac{\cosh(rx)}{\cosh x}
  \ge(1-r^2)(1-\sech x)
  \ge(1-r^2)\mathcal L(x),
\]
where the last inequality is \eqref{eq:L-upper}.

For \eqref{eq:hyp-two}, the cases $x=0$ or $r=0$ are equalities.  Assume
$x,r>0$ and define
\begin{equation}\label{eq:h-definition}
  h(y):=\frac{\sinh^2y}{y^2\cosh y},
  \qquad y>0,
  \qquad h(0):=1.
\end{equation}
The sign of $h'(y)$ is the sign of
\begin{equation}\label{eq:N-definition}
  N(y):=y(2+\sinh^2y)-2\sinh y\cosh y,
\end{equation}
because
\[
  \frac{h'(y)}{h(y)}=2\coth y-\frac2y-\tanh y.
\]
Moreover,
\[
  N'(y)=y\sinh y\cosh y
  \left(2-3\frac{\tanh y}{y}\right).
\]
Since $\sinh y\cosh y>y$ for $y>0$, the function $\tanh y/y$
is strictly decreasing, with limits $1$ at zero and $0$ at infinity.
Thus $N'$ changes sign exactly once, from
negative to positive. Since $N(0)=0$ and $N(y)\to\infty$ as
$y\to\infty$, the same is true of $N$ on $(0,\infty)$.
Consequently $h$ first decreases and
then increases.  Its maximum on $[0,x]$ is therefore attained at an endpoint,
and
\begin{equation}\label{eq:h-endpoint}
  \frac{\sinh^2(rx)}{r^2\cosh(rx)}
  \le
  \max\left\{x^2,\frac{\sinh^2x}{\cosh x}\right\}.
\end{equation}
By \eqref{eq:L-lower-one} and \eqref{eq:L-lower-two},
\[
  2\mathcal L(x)\cosh x\ge x^2,
  \qquad
  2\mathcal L(x)\cosh x\ge\frac{\sinh^2x}{\cosh x}.
\]
Combining these inequalities with \eqref{eq:h-endpoint}, and multiplying by
$r^2\cosh(rx)$, proves \eqref{eq:hyp-two}.
\end{proof}

\begin{proof}[Proof of Lemma \ref{lem:nonlinear-K}]
First suppose $w\ne\pm1$.  Choose a branch of the logarithm and write
\begin{equation}\label{eq:tanh-strip}
  \xi=a+i\vartheta:=\frac12\log\frac{1+w}{1-w},
  \qquad
  w=\tanh\xi.
\end{equation}
The real part $a$ does not depend on the branch. Indeed,  changing the
logarithmic branch replaces $\vartheta$ by
$\vartheta+\pi k$, $k\in\mathbb Z$; hence all subsequent expressions
involving $\cos(2\vartheta)$ and $\sin(2\vartheta)$ are
branch-independent. Since
\[
  1+\tanh\xi=\frac{e^\xi}{\cosh\xi},
  \qquad
  1-\tanh\xi=\frac{e^{-\xi}}{\cosh\xi},
\]
we obtain the exact strip formula
\begin{equation}\label{eq:H-strip}
  H_s(\tanh\xi)
  =\frac1s\log\cosh(sa)-\log\abs{\cosh\xi}.
\end{equation}
Since $G_s=2\partial_{\overline w}H_s$, differentiating
\eqref{eq:H-strip} with respect to $\overline w$, with $s$ fixed,
gives 
\begin{equation}\label{eq:G-strip}
  G_s(w)=\frac{\tanh(sa)-\overline w}{1-\overline w^{\,2}}.
\end{equation}
Indeed,
\[
  \frac{\partial a}{\partial\overline w}
  =\frac1{2(1-\overline w^{\,2})},
  \qquad
  2\partial_{\overline w}\log\abs{\cosh\xi}
  =\frac{\overline w}{1-\overline w^{\,2}}.
\]
Multiplying \eqref{eq:G-strip} by $\overline w$ and using the hyperbolic
addition formulas yields
\begin{equation}\label{eq:B-strip}
  \overline wG_s(w)
  =\frac12\left(
  1-\frac{\cosh((s-2)a+2i\vartheta)}{\cosh(sa)}
  \right).
\end{equation}
For clarity, this last identity follows from
\begin{align*}
  \frac{\tanh u\,[\tanh v-\tanh u]}{1-\tanh^2u}
  &=\frac{\sinh u\sinh(v-u)}{\cosh v},\\
  2\sinh u\sinh(v-u)&=\cosh v-\cosh(v-2u),
\end{align*}
with $u=\overline\xi$ and $v=sa$.

Because $H_s$ is even, $K_s(-w)=K_s(w)$; replacing $w$ by $-w$ if
necessary, assume $a\ge0$.  Put
\begin{equation}\label{eq:x-r-CS-gamma}
\begin{aligned}
  x&:=sa,
  &r&:=\frac{s-2}{s}\in(0,1),
  &\gamma&:=\cos(2\vartheta),\\
  C&:=\frac{\cosh(rx)}{\cosh x},
  &S&:=\frac{\sinh(rx)}{\cosh x}.
\end{aligned}
\end{equation}
Since $s=2/(1-r)$, differentiation of \eqref{eq:H-strip} in $s$, followed
by taking real and imaginary parts in \eqref{eq:B-strip}, gives
\begin{align}
  \mathcal E_s(w)&=\frac{(1-r)^2}{4}\mathcal L(x),\label{eq:E-formula}\\
  \mathcal R_s(w)&=\frac12(1-C\gamma),\label{eq:R-formula}\\
  \mathcal I_s(w)&=-\frac12S\sin(2\vartheta),\label{eq:I-formula}\\
  \frac{\mathcal R_s(w)}2-(s-1)\mathcal E_s(w)
  &=\frac14\left(1-C\gamma-(1-r^2)\mathcal L(x)\right).
  \label{eq:K22-formula}
\end{align}
If $x=0$, then $\mathcal E_s=\mathcal I_s=0$ and
$\mathcal R_s=(1-\cos(2\vartheta))/2\ge0$, so $K_s(w)\succeq0$.
Assume $x>0$.  The upper-left entry of $K_s(w)$ is then positive.  Using $s-2=2r/(1-r)$, the determinant is
\begin{equation}\label{eq:detK}
  \det K_s(w)=\frac{(1-r)^2}{64r^2}\,\mathcal Q(\gamma),
\end{equation}
where
\begin{equation}\label{eq:Q-gamma}
  \mathcal Q(\gamma)
  :=4r^2\mathcal L(x)
  \left(1-(1-r^2)\mathcal L(x)-C\gamma\right)
  -S^2(1-\gamma^2).
\end{equation}
This is a quadratic polynomial in $\gamma$ with positive leading coefficient
$S^2$ and vertex
\begin{equation}\label{eq:gamma-star}
  \gamma_*
  =\frac{2r^2\mathcal L(x)C}{S^2}
  =\frac{2r^2\mathcal L(x)\cosh(rx)\cosh x}{\sinh^2(rx)}.
\end{equation}
By \eqref{eq:hyp-two}, $\gamma_*\ge1$.  Hence $\mathcal Q$ is
nonincreasing on $[-1,1]$, and its minimum there is attained at $\gamma=1$.
At that point,
\[
  \mathcal Q(1)
  =4r^2\mathcal L(x)
  \left(1-C-(1-r^2)\mathcal L(x)\right)\ge0
\]
by \eqref{eq:hyp-one}.  Thus $\det K_s(w)\ge0$.  Since $K_{11}>0$,
we also have $K_{22}\ge K_{12}^2/K_{11}\ge0$, and therefore
$K_s(w)\succeq0$.  The points $w=\pm1$ follow by continuity.
\end{proof}

\begin{proof}[Proof of \Cref{thm:infinitesimal}]
Let $\lambda=-c/2+i\beta$.  Since the directional derivative of a real
function with complex gradient $G_s$ in direction $\delta w$ is
$\Rea(G_s(w)\overline{\delta w})$, equations \eqref{eq:ERI} give
\begin{equation}\label{eq:directional-derivative}
  \left.\frac{\dd}{\dd t}\right|_{t=0}
  H_{s+\kappa t}(e^{\lambda t}w)
  =\kappa\mathcal E_s(w)-\frac c2\mathcal R_s(w)
   +\beta\mathcal I_s(w).
\end{equation}
A direct multiplication of \eqref{eq:K-matrix} and \eqref{eq:X-matrix}
shows that
\begin{align}
  \tr\bigl(K_s(w)X_s(c,\beta,\kappa)\bigr)
  &=\frac c2\mathcal R_s(w)-\kappa\mathcal E_s(w)
    -\beta\mathcal I_s(w)\notag\\
  &=-\left.\frac{\dd}{\dd t}\right|_{t=0}
  H_{s+\kappa t}(e^{\lambda t}w).\label{eq:trace-pairing}
\end{align}
Both matrices in the trace pairing are positive semidefinite, so
\[
  \tr(K_sX_s)=\tr(K_s^{1/2}X_sK_s^{1/2})\ge0.
\]
Equation \eqref{eq:trace-pairing} proves \eqref{eq:inf-derivative}.
\end{proof}

We shall use the following immediate pathwise form.  Let
$s\in C^1(I;(2,\infty))$, let
$\mu\in C^1(I;\C\setminus\{0\})$, and write
\begin{equation}\label{eq:path-log-derivative}
  \frac{\mu'(t)}{\mu(t)}=-\frac{c(t)}2+i\beta(t),
  \qquad
  \kappa(t):=s'(t).
\end{equation}
If
\begin{equation}\label{eq:path-cone-condition}
  X_{s(t)}\bigl(c(t),\beta(t),\kappa(t)\bigr)\succeq0
  \qquad(t\in I),
\end{equation}
then \Cref{thm:infinitesimal} and the chain rule show that, for every fixed
$w\in\C$,
\begin{equation}\label{eq:pathwise-monotonicity}
  t\longmapsto H_{s(t)}(\mu(t)w)
  \quad\text{is nonincreasing on }I.
\end{equation}
The path $s(t)$ itself need not be monotone.

\subsection{Step 3: the radial boundary and a smooth rank-one flow}

Fix $1<P\le Q$.  Retain $\tau_R=\frac12\log R$ and put
$m_P:=\sinh\tau_P$.  For $0<m\le m_P$ and $R\ge P$, define
\begin{equation}\label{eq:chi-definition}
  \chi_R:=\operatorname{arcosh}
  \frac{\cosh\tau_R}{\sqrt{1+m^2}},
  \qquad
  \phi_R:=\arctan\frac{\tanh\chi_R}{m}\in[0,\pi/2).
\end{equation}
At $m=0$, use the continuous extensions $\chi_R=\tau_R$ and
$\phi_R=\pi/2$.  The definitions give
\begin{equation}\label{eq:boundary-coordinates}
\begin{aligned}
  \cosh\tau_R&=\sqrt{1+m^2}\cosh\chi_R,\\
  \sinh\tau_R\cos\phi_R&=m\cosh\chi_R,
  &\sinh\tau_R\sin\phi_R&=\sinh\chi_R.
\end{aligned}
\end{equation}
Indeed, the squares of the last two right sides sum to
$\sinh^2\tau_R$.  For $\sigma\in\{-1,+1\}$ set
\begin{equation}\label{eq:z-sigma}
  z_\sigma(m):=
  \left(\frac PQ\right)^{1/4}
  \exp\left[-\frac{\chi_Q+\sigma\chi_P}{2}
       +i\frac{\phi_Q+\sigma\phi_P}{2}\right].
\end{equation}

\begin{proposition}[Boundary parametrization]\label{prop:boundary-parametrization}
The points $z_\sigma(m)$, with $0\le m\le m_P$ and
$\sigma\in\{-1,+1\}$, cover the radial boundary of $\Om_{p,q}$
in the first quadrant.
\end{proposition}

\begin{proof}
The addition formulas and \eqref{eq:boundary-coordinates} give
\begin{equation}\label{eq:addition-key}
\begin{aligned}
  &\cosh\tau_P\cosh\tau_Q
   -\sinh\tau_P\sinh\tau_Q\cos(\phi_Q+\sigma\phi_P)\\
  &\hspace{10em}=\cosh(\chi_Q+\sigma\chi_P).
\end{aligned}
\end{equation}
Thus, with $\theta_\sigma=(\phi_Q+\sigma\phi_P)/2$ and
$y_\sigma=|z_\sigma(m)|^2$, equation \eqref{eq:A-hyperbolic} yields
\[
  Qy_\sigma+P/y_\sigma=A_{P,Q}(\theta_\sigma).
\]
Hence $y_\sigma$ solves the boundary quadratic
\eqref{eq:boundary-quadratic}.  Since $\chi_Q\ge\chi_P\ge0$,
\[
  y_\sigma=\sqrt{P/Q}\,e^{-(\chi_Q+\sigma\chi_P)}
  \le\sqrt{P/Q},
\]
so it is the smaller root, allowing equality when the roots coincide.
Lemma~\ref{lem:radial-boundary} therefore places $z_\sigma(m)$ on
the radial boundary, including the point $z=1$ when $P=Q$.

The angles are continuous in $m$ and lie in $[0,\pi/2]$ because
$0\le\phi_P\le\phi_Q\le\pi/2$.  At $m=0$ they are
$\theta_-=0$ and $\theta_+=\pi/2$.  At $m=m_P$ one has
$\chi_P=\phi_P=0$, so both angles equal $\phi_Q/2$.
The intermediate value theorem proves coverage.
\end{proof}

We now reach both boundary branches by the same smooth flow.
Fix $m\in(0,m_P]$, write $B=\sqrt{1+m^2}$, and define for real $t$
\begin{equation}\label{eq:signed-flow-coordinates}
  \tau(t):=\operatorname{arcosh}(B\cosh t),
  \quad R(t):=e^{2\tau(t)},
  \quad \phi(t):=\arctan\frac{\tanh t}{m}.
\end{equation}
These functions are smooth, with $R(t)>1$.  For a fixed sign $\sigma$,
let $t_0=-\sigma\chi_P$, $t_1=\chi_Q$, and put
\begin{equation}\label{eq:eta-path}
  \mu(t):=\left(\frac{P}{R(t)}\right)^{1/4}
  \exp\left[-\frac{t-t_0}{2}
       +i\frac{\phi(t)-\phi(t_0)}{2}\right],
  \qquad t_0\le t\le t_1.
\end{equation}
The endpoint identities are
\begin{equation}\label{eq:flow-endpoints}
  R(t_0)=P,\quad R(t_1)=Q,\quad
  \mu(t_0)=1,\quad \mu(t_1)=z_\sigma(m).
\end{equation}

Set $v=B\sinh t/\sinh\tau(t)$.  Since
\begin{equation}\label{eq:flow-v-identity}
  \sinh^2\tau(t)=m^2+B^2\sinh^2t,
  \qquad 1-v^2=\frac{m^2}{\sinh^2\tau(t)}>0,
\end{equation}
differentiation gives, with $R=R(t)$ and $\tau=\tau(t)$,
\begin{equation}\label{eq:flow-substitution}
\begin{aligned}
  s&=R+1, &\kappa=s'&=2Rv,\\
  c&=-2\Rea\frac{\mu'}\mu=1+v,
  &\beta=\Ima\frac{\mu'}\mu&=\frac{m}{2\sinh^2\tau}.
\end{aligned}
\end{equation}
Here $\phi'=m/\sinh^2\tau$, as follows directly from
\eqref{eq:signed-flow-coordinates}.  Consequently,
\begin{equation}\label{eq:rank-one-cone}
  X_s(c,\beta,\kappa)=
  \begin{pmatrix}
    R(1-v)&\sqrt R\,m/\sinh\tau\\
    \sqrt R\,m/\sinh\tau&1+v
  \end{pmatrix}\succeq0,
  \qquad \det X_s(c,\beta,\kappa)=0.
\end{equation}
The diagonal entries are positive and the determinant vanishes by
\eqref{eq:flow-v-identity}.  For the plus branch with $m<m_P$, $s(t)$
first decreases and then increases, staying strictly above two; the flow is smooth at
its turning point $t=0$.  This is why the infinitesimal inequality is
used with either sign of $\kappa$.

\subsection{Step 4: integration of the flow}

Write $f(x)=a+bx$ on the two-point space, so that
$T_zf(x)=a+zbx$.  First let $z=z_\sigma(m)$ with $m>0$ and use
the flow above.  When $t_0=t_1$, the endpoint is $z=1$ and $p=q$,
so the desired inequality is an equality.  Otherwise, if $a\ne0$,
\[
  \log\|T_{\mu(t)}f\|_{s(t)}
  =\log|a|+H_{s(t)}\!\left(\mu(t)b/a\right)
\]
is nonincreasing by \eqref{eq:rank-one-cone} and the pathwise
inequality \eqref{eq:pathwise-monotonicity}.  If $a=0$ and $b\ne0$,
the norm is $|\mu(t)b|$, whose logarithmic derivative is
$-c/2<0$; the zero function is immediate.  Applying
\eqref{eq:flow-endpoints} gives
\[
  \|T_{z_\sigma(m)}f\|_q\le\|f\|_p.
\]
Continuity in $m$ gives the same inequality at $m=0$.
Proposition~\ref{prop:boundary-parametrization} proves it on the
first-quadrant boundary.  Replacing $z$ by $-z$ interchanges the output
values, and complex conjugation handles $\overline z$; hence the
inequality holds on the entire radial boundary.

Finally, every nonzero interior point is $z=r z_0$ with $0<r<1$ and
$z_0$ on the radial boundary.  The real multiplier is a Markov
contraction, since
\[
  T_rh(x)=\frac{1+r}{2}h(x)+\frac{1-r}{2}h(-x)
\]
and Jensen's inequality gives $\|T_rh\|_q\le\|h\|_q$.
Using $T_uT_v=T_{uv}$, we conclude that
\[
  \|T_zf\|_q\le\|T_{z_0}f\|_q\le\|f\|_p.
\]
The case $z=0$ follows from Jensen's inequality as well.
This proves \Cref{thm:two-point}, including the diagonal case $p=q$.

\subsection{Step 5: the dual range below two}

For $1<s<\infty$, let $s'=s/(s-1)$.  Then
\begin{equation}\label{eq:D-duality}
  D_{s'}=D_s^{-1},
  \qquad
  M_{\overline z}=M_z^{\mathsf T}.
\end{equation}
We first record the exact duality of the admissible regions:
\begin{equation}\label{eq:Omega-duality}
  z\in\Om_{p,q}
  \quad\Longleftrightarrow\quad
  \overline z\in\Om_{q',p'}.
\end{equation}
Indeed, with
\[
  B:=D_q^{1/2}M_zD_p^{-1/2},
\]
Lemma \ref{lem:matrix-domain} says that $z\in\Om_{p,q}$ precisely when
$B^{\mathsf T}B\preceq I$.  A matrix and its transpose have the same
operator norm, so this is equivalent to $BB^{\mathsf T}\preceq I$, or
\[
  M_zD_p^{-1}M_z^{\mathsf T}\preceq D_q^{-1}.
\]
Using \eqref{eq:D-duality}, this is exactly
\[
  M_{\overline z}^{\mathsf T}D_{p'}M_{\overline z}
  \preceq D_{q'},
\]
which proves \eqref{eq:Omega-duality} by Lemma \ref{lem:matrix-domain}.

Now suppose $1<p\le q<2$ and $z\in\Om_{p,q}$. Then
\[
  2<q'\le p'<\infty,
  \qquad
  \overline z\in\Om_{q',p'}.
\]
The result already proved above, applied to the pair $(q',p')$ and
multiplier $\overline z$, gives
\[
  \norm{T_{\overline z}h}_{p'}\le\norm h_{q'}.
\]
With respect to normalized counting measure on the two-point space,
$T_z^*=T_{\overline z}$.  Hence
\begin{align*}
  \norm{T_zf}_q
  &=\sup_{\norm h_{q'}=1}\abs{\ip{T_zf}{h}}\\
  &=\sup_{\norm h_{q'}=1}\abs{\ip f{T_{\overline z}h}}\\
  &\le\norm f_p.
\end{align*}
This proves the two-point inequality throughout $1<p\le q<2$.
Together with tensorization and the known ranges, it completes the
characterization of complex hypercontractivity on the Hamming cube.

\section*{Acknowledgments}
 Y.~C.~H.~would like to thank Yanqi Qiu for helpful support over the years.
 P.~I.~acknowledges partial support from the US NSF CAREER grant DMS-2152401, US NSF grant DMS-2554183,  a Simons Fellowship, and a Humboldt Research Fellowship for Experienced Researchers. The authors acknowledge the use of AI tools. All mathematical arguments and proofs in the final manuscript were checked and written by the authors.

\end{document}